%% file: main.tex
\documentclass{article}

\usepackage{arxiv}

\usepackage[utf8]{inputenc} 
\usepackage[T1]{fontenc}    
\usepackage{url}            
\usepackage{booktabs}       
\usepackage{amsfonts}       
\usepackage{nicefrac}       
\usepackage{microtype}      
\usepackage{lipsum}
\usepackage{verbatim}       

\RequirePackage{amssymb}    
\RequirePackage{amsthm}     
\RequirePackage{thmtools}   
\RequirePackage{mathtools}  
\RequirePackage{mathrsfs}   
\RequirePackage{cancel}     
\RequirePackage{stmaryrd}   

\RequirePackage{tikz}             
\usetikzlibrary{calc}
\usetikzlibrary{intersections}
\usetikzlibrary{decorations.markings}
\usetikzlibrary{cd} 
\RequirePackage[all]{xy} 

\RequirePackage{varioref}                                     
\RequirePackage{hyperref}                                     
\RequirePackage[nameinlink, capitalize, noabbrev]{cleveref}   
\RequirePackage{doi}           

\declaretheoremstyle[headfont   = \bfseries\sffamily,
                     notefont   = \normalfont,
                     spaceabove = 6pt plus 0pt minus 2pt]{plain}
\declaretheoremstyle[headfont   = \bfseries\sffamily,
                     notefont   = \normalfont,
                     spaceabove = 6pt plus 0pt minus 2pt]{definition}
\declaretheorem[style = plain, numberwithin = section]{theorem}
\declaretheorem[style = plain,      sibling = theorem]{corollary}
\declaretheorem[style = plain,      sibling = theorem]{lemma}
\declaretheorem[style = plain,      sibling = theorem]{proposition}

\declaretheorem[style = definition, sibling = theorem]{definition}
\declaretheorem[style = definition, sibling = theorem]{example}
\declaretheorem[style = definition, sibling = theorem]{notation}
\declaretheorem[style = remark,     sibling = theorem]{remark}

\crefname{observation}{Observation}{Observations}
\Crefname{observation}{Observation}{Observations}
\crefname{conjecture}{Conjecture}{Conjectures}
\Crefname{conjecture}{Conjecture}{Conjectures}
\crefname{notation}{Notation}{Notations}
\Crefname{notation}{Notation}{Notations}
\crefname{paper}{Paper}{Papers}
\Crefname{paper}{Paper}{Papers}

\newcommand{\x}{\mathbf{x}}
\newcommand{\y}{\mathbf{y}}

\title{Iterative Desingularization}

\author{
  Vegard Fjellbo \\
  Department of Mathematics \\
  University of Oslo \\
  Oslo, Norway \\
  \texttt{vegard.fjellbo@gmail.com} \\
   \And
}

\begin{document}
\maketitle

\input{sections/abstract.tex}


\input{sections/intro.tex}

\input{sections/prelim.tex}

\input{sections/methods.tex}

\input{sections/calc.tex}

\input{sections/descr.tex}

\bibliographystyle{unsrt}  
\bibliography{main}

\end{document}

%% file: sections/abstract.tex
\begin{abstract}
A simplicial set is said to be \textbf{non-singular} if the representing map of each non-degenerate simplex is degreewise injective. The inclusion into the category of simplicial sets, of the full subcategory whose objects are the non-singular simplicial sets, admits a left adjoint functor called desingularization. In this paper, we provide an iterative description of desingularization that is useful for theoretical purposes as well as for doing calculations.

\vspace{1em}
MSC-class: 55U10 (Primary), 18A40 (Secondary)
\end{abstract}

%% file: sections/intro.tex
\section{Introduction}
\label{sec:intro_itdesing}

Desingularization is defined thus.
\begin{definition}\label{def:desingularization}
Let $X$ be a simplicial set. The \textbf{desingularization} of $X$ \cite[Rem.~2.2.12]{WJR13}, denoted $DX$, is the image of the map
\[X\to \prod _{f:X\rightarrow Y}Y\]
given by $x\mapsto (f(x))_f$, where the product is indexed over the quotient maps $f:X\rightarrow Y$ whose targets $Y$ are non-singular.
\end{definition}
\noindent A product of non-singular simplicial sets is again non-singular \cite[Rem.~2.2.12]{WJR13} and a simplicial subset of a non-singular simplicial set is again non-singular \cite[Rem.~2.2.12]{WJR13}. Therefore, the simplicial set $DX$ is non-singular. In this paper, we will give a systematic, but minimal introduction to the functor $D$.

If we corestrict the map $X\to \prod _{f:X\rightarrow Y}Y$ to its image $DX$, then we get a map $\eta _X:X\to DX$. By this we simply mean the following. If $h:Z\to W$ is a simplicial map whose image is contained in some simplicial subset $W'$ of $W$, then we say that the induced map $Z\to W'$ is a \textbf{corestriction} of $h$ to $W'$.

Thus far, the description given in \cref{def:desingularization} is the only description available in the literature. In this paper, we provide the viewpoint of \cref{thm:main_result_itdesing} to desingularization. To obtain this viewpoint, we introduce the notion of enforcer in \cref{def:enforcer}.

When we say that a simplex is \textbf{embedded} if its representing map is degreewise injective, we get a more convenient definition of the term \emph{non-singular simplicial set}. Given a simplicial set $X$ and a non-degenerate simplex $x$ in $X$, the enforcer $\rho _x$ is the degeneracy operator that in the least drastic way makes the cobase change of the representing map of $x$ into the representing map of a degenerate simplex, in the case when $x$ is not embedded, or that makes the trivial cobase change, in the case when $x$ is embedded. In other words, the enforcer is the degeneracy operator that is as close as possible to the identity meanwhile honouring any pairwise equalities between the vertices of $x$.

Simultaneously pushing out along all the enforcers associated with a simplicial set $X$ yields a simplicial set $Cen(X)$ that we refer to as the enforced collapse of $X$. The notion is properly introduced in \cref{def:enforced_collapse}. One should think of the enforced collapse as a preferred first step towards making $X$ non-singular. If some non-degenerate simplex of $X$ is not embedded, then we say that $X$ is \textbf{singular}. Note that $Cen(X)$ may be singular. By \cref{lem:pushout_along_enforcers_intermediate_desingularization}, which is formulated in a slightly generalized context compared with the enforced collapse, we get that pushing out along enforcers is never too drastic. Moreover, if the result is non-singular, then it is canonically the desingularization.

We are ready to explain the iterative description of desingularization, which is formulated using the following piece of language.
\begin{definition}\label{def:sequence_composition}
Let $\mathscr{C}$ be some cocomplete category and suppose $\lambda$ some ordinal. A \textbf{$\lambda$-sequence in $\mathscr{C}$} is a cocontinous functor $X:\lambda \to \mathscr{C}$, that we will denote
\begin{displaymath}
\xymatrix{
X^{[0]} \ar[r]^{f^{0,1}} & X^{[1]} \ar[r]^{f^{1,2}} & \cdots \ar[r] & X^{[\beta ]}\ar[r]^{f^{\beta ,\beta +1}} & \cdots \, .
}
\end{displaymath}
The canonical map $X^{[0]}\to colim_{\beta <\lambda }X^{[\beta ]}$ is the \textbf{composition} of the $\lambda$-sequence. A \textbf{sequence} in $\mathscr{C}$ is a $\lambda$-sequence for some $\lambda$.
\end{definition}
\noindent If $\lambda$ is finite, then the composition is a composite in the usual sense.
\begin{theorem}\label{thm:main_result_itdesing}
Let $X$ be a simplicial set. There is an ordinal $\lambda$ such that the map $\eta _X:X\to UDX$ is the composition of the $\lambda$-sequence
\begin{displaymath}
\xymatrix{
Cen^0(X) \ar[r] & Cen^1(X) \ar[r] & \cdots \ar[r] & Cen^\beta (X) \ar[r] & \cdots \, .
}
\end{displaymath}
of iterations of the enforced collapse.
\end{theorem}
\noindent \cref{thm:main_result_itdesing} provides an alternative description of the desingularization functor. Note that the ordinal $\lambda$ depends on the simplicial set $X$.

Let $sSet$ denote the category of simplicial sets. Furthermore, let $nsSet$ denote the category of non-singular simplicial sets. It is by definition the full subcategory of $sSet$ whose objects are the non-singular simplicial sets.

As we explain \cref{def:desingularization} and as we explain how desingularization is functorial in \cref{sec:prelim}, we fix some notation and terminology to be used throughout the paper. Furthermore, we point out the implications for limits and colimits in $nsSet$ of the fact that $D$ is left adjoint to the (full) inclusion $U:nsSet\to sSet$. \cref{sec:prelim} is merely an elaboration of \cite[Rem.~2.2.12]{WJR13}, where desingularization is introduced.

In \cref{sec:calculations}, we introduce the enforcer to serve as the most basic technology for doing calculations as well as for theory. Building on this notion, we provide the two results \cref{prop:role_of_enforcers} and \cref{lem:pushout_along_enforcers_intermediate_desingularization} as tools.

We illustrate how desingularization behaves in \cref{sec:examples}. Our examples include applying $D$ to highly singular, somewhat subdivided and very subdivided simplicial sets, most of which are models of low-dimensional spheres.

Finally, in \cref{sec:description}, we explain how \cref{prop:role_of_enforcers} and \cref{lem:pushout_along_enforcers_intermediate_desingularization} can be used to construct the sequence that \cref{thm:main_result_itdesing} refers to and we conclude the section as well as the paper by deducing \cref{thm:main_result_itdesing} from the construction.

%% file: sections/prelim.tex
\section{Preliminaries}
\label{sec:prelim}

In this section, we establish the functorality of desingularization. To do this, we first fix some basic notation and terminology, which is anyhow useful throughout this paper. Additionally, we properly explain \cref{def:desingularization} to avoid any confusion.

\subsection{Notation and terminology}

Fritsch and Piccinini \cite{FP90} is a source of the style we use, when it comes to notation and terminology.

The category
\[sSet=Fun(\Delta ^{op},Set)\]
is the category of functors (and natural transformations) with source $\Delta ^{op}$ and target the category $Set$ of sets (and functions). When we write $\Delta$, we mean the skeleton of finite ordinals whose objects are totally ordered sets
\[[n]=\{ 0<1<\dots <n\}\]
and whose morphisms are order-preserving functions $\alpha :[m]\to [n]$, meaning $\alpha (i)\leq \alpha (j)$ if $i\leq j$. An object in the category $sSet$ is a \textbf{simplicial set}.

Morphisms of $\Delta$ are referred to as \textbf{operators}. We sometimes think of a simplicial set $X$ as an $\mathbb{N} _0$-graded set $\bigsqcup _{n\geq 0}X_n$ with operators acting from the right. Here, we mean $X_n=X([n])$, $n\geq 0$. Elements of $X_n$ are referred to as \textbf{$n$-simplices}, $n\geq 0$. We also say that $n$ is the \textbf{degree} of $x$ if $x$ is an $n$-simplex. If $x$ is an $n$-simplex of $X$ and if $\alpha :[m]\to [n]$ is an operator, then $\alpha$ acts on $x$ from the right. The result will be denoted $x\alpha$. The induced function $\alpha ^*:X_n\to X_m$ thus takes $x$ to $\alpha ^*(x)=x\alpha$.

When we think of simplicial sets as graded sets under right action of operators, we also think of a simplicial map $f:X\to Y$, meaning a natural transformation $X\Rightarrow Y$, as a function that respect the degree, meaning $f(x)\in Y_n$ if $x\in X_n$, and that is compatible with the right action of operators, meaning $f(x\alpha )=f(x)\alpha$.

An operator $\alpha :[m]\to [n]$ is referred to as a \textbf{face operator} if $\alpha (i)\neq \alpha (j)$ whenever $i,j\in \{0,\dots ,m\}$ and $i\neq j$. It is referred to as a \textbf{degeneracy operator} if $k=\alpha (j)$ for some $j\in \{0,\dots ,m\}$ for all $k\in \{0,\dots ,n\}$. These classes of morphisms are precisely the monomorphisms and epimorphisms of $\Delta$, respectively.

For each $n>0$ and each $j$ with $0\leq j\leq n$, we can define the face operator $\delta ^n_j:[n-1]\to [n]$ such that $j$ is not in its image, referred to as an \textbf{elementary face operator}. Similarly, for each $n\geq 0$, we can define the degeneracy operator $\sigma ^n_j:[n+1]\to [n]$ with $j\mapsto j$ and $j+1\mapsto j$. Also useful is the \textbf{vertex operator} $\varepsilon ^n_j:[0]\to [n]$ with $0\mapsto j$, defined whenever $0\leq j\leq n$. We often omit the upper index when referring to these three special types of operators.

A degeneracy operator or face operator is \textbf{proper} if it is not an identity morphism. We say that a simplex $y$ is a \textbf{(proper) face} of a simplex $x$ if $y=x\mu$ for some (proper) face operator $\mu$ and that $y$ is a \textbf{(proper) degeneracy} of $x$ if $y=x\rho$ for some (proper) degeneracy operator $\rho$. A simplex is \textbf{non-degenerate} if it is not a proper degeneracy. 

The Eilenberg-Zilber lemma \cite[Thm.~4.2.3]{FP90} says that any simplex $x$ of a simplicial set can be written uniquely as a degeneration of a non-degenerate simplex. This means that there is a unique pair $(x^\sharp ,x^\flat )$ consisting of a non-degenerate simplex $x^\sharp$ and a degeneracy operator $x^\flat$ that satisfies
\[x=x^\sharp x^\flat.\]
The non-degenerate simplex $x^\sharp$ will be referred to as the \textbf{non-degenerate part} of $x$ and $x^\flat$ will be referred to as the \textbf{degenerate part} of $x$. We let $X^\sharp$ denote the set of non-degenerate simplices of a simplicial set $X$ and $X^\sharp _n$ the set of non-degenerate simplices of degree $n$, for each $n\geq 0$.

By the Yoneda lemma, there is a natural bijective correspondence $x\mapsto \bar{x}$ between the set $X_n$ of $n$-simplices and the set of simplicial maps $\Delta [n]\to X$. We say that
\[\bar{x}:\Delta [n]\to X\]
is the \textbf{representing map} of the simplex $x$.

\subsection{Quotients}

Desingularization has the following property \cite[Rem.~2.2.12]{WJR13}.
\begin{lemma}\label{lem:desing_unique_factorization_through_unit}
Let $X$ be a simplicial set. Every simplicial map whose source is $X$ and whose target is non-singular factors uniquely through $\eta _X$.
\end{lemma}
\noindent Before we prove the property, we explain \cref{def:desingularization} properly.

Let $X$ be some simplicial set. Consider the event that we for each $n\geq 0$ have an equivalence relation $R_n$ on $X_n$ such that whenever we have an operator $\alpha :[m]\to [n]$, then the composite
\[R_n\to X_n\times X_n\xrightarrow{\alpha ^*\times \alpha ^*} X_m\times X_m\]
corestricts to $R_m\subseteq X_m\times X_m$. This means that we have a commutative square
\begin{equation}
\label{eq:first_diagram_proof_lem_desing_unique_factorization_through_unit}
\begin{gathered}
\xymatrix{
R_n \ar[d]_{\bar{\alpha } } \ar[r] & X_n\times X_n \ar[d]^{\alpha ^*\times \alpha ^*} \\
R_m \ar[r] & X_m\times X_m
}
\end{gathered}
\end{equation}
which in turn gives rise to a dashed map in the square
\begin{equation}
\label{eq:second_diagram_proof_lem_desing_unique_factorization_through_unit}
\begin{gathered}
\xymatrix{
X_n \ar[d]_{\alpha ^*} \ar[r] & X_n/R_n \ar@{-->}[d] \\
X_m \ar[r] & X_m/R_m
}
\end{gathered}
\end{equation}
such that it commutes.

Thus we obtain a simplicial set $X/R$ given by defining the set
\[(X/R)_n=X_n/R_n\]
as the set of $n$-simplices. It is readily checked that the right hand vertical map in (\ref{eq:second_diagram_proof_lem_desing_unique_factorization_through_unit}) is a right action of the operator $\alpha$ on the set $X_n/R_n$ so that $X/R$ is indeed a simplicial set. From the commutativity of (\ref{eq:second_diagram_proof_lem_desing_unique_factorization_through_unit}), it is automatic that the canonical map $X\to X/R$ is a simplicial map. We say that it is a \textbf{quotient map}. If we fix a simplicial set $X$, then the quotient maps $X\to Y$ form a set. This explains \cref{def:desingularization}.

If $f:X\to Y$ is a degreewise surjective simplicial map, then we may define $R_n$, $n\geq 0$, by letting $x\sim x'$ if $f(x)=f(x')$. Because $f$ respects operators, as a simplicial map, it follows that the equivalence relations $R_n$, $n\geq 0$, form a set of equivalence relations of the type described above. By making a choice of a representative one can define a map $X/R\to Y$ such that the triangle
\begin{equation}
\label{eq:third_diagram_proof_lem_desing_unique_factorization_through_unit}
\begin{gathered}
\xymatrix{
X \ar[dr] \ar[rr]^f && Y \\
& X/R \ar@{-->}[ur]_\cong
}
\end{gathered}
\end{equation}
commutes. The dashed map is an isomorphism by design. This makes \cref{def:desingularization} meaningful in the sense that we can obtain \cref{lem:desing_unique_factorization_through_unit}. 

We are ready to prove the lemma.
\begin{proof}[Proof of \cref{lem:desing_unique_factorization_through_unit}.]
Let $k:X\to A$ be a map whose target $A$ is non-singular. First, note that there is at most one map $\bar{k}$ such that $k=\bar{k} \circ \eta _X$. This is because $\eta _X$ is degreewise surjective and because the degreewise surjective maps are precisely the epimorphisms of $sSet$ \cite[p.~142]{FP90}. It remains to argue that there is a map $\bar{k}$ such that $k=\bar{k} \circ \eta _X$.

Corestrict $k$ to its image $A'$ so that we get a factorization
\begin{displaymath}
 \xymatrix{
 X \ar[d]_{k''} \ar[dr]^{k'} \ar[r]^k & A \\
 X/R \ar[r]_(.35)\cong & A'=\textrm{Im } k \ar[u]_h
 }
\end{displaymath}
of $k$. Then the map $k'$ is a degreewise surjective map whose target is non-singular. We get the diagram
\begin{equation}
\label{eq:fourth_diagram_proof_lem_desing_unique_factorization_through_unit}
\begin{gathered}
\xymatrix{
& \prod _{f:X\rightarrow Y}Y \ar[dr]^{pr_{k''}} \\
X \ar[ur]^{x\mapsto (f(x))_f} \ar[dr]_{\eta _X} \ar[rr]^(.35){k''} & \ar[u] & X/R \\
& DX \ar@{-}[u] \ar[ur]_g
}
\end{gathered}
\end{equation}
in which we have restricted the projection map
\[\textrm{pr} _{k''}:\prod _{f:X\rightarrow Y}Y\to X/R\]
to $DX$ --- a restriction we denote $g$.

From (\ref{eq:fourth_diagram_proof_lem_desing_unique_factorization_through_unit}) we can conclude that $k''=g\circ \eta _X$ as the outer square and the upper triangle commute. Hence, by the design of $DX$, the map $k'$ factors through the restriction $g$ up to identification with a quotient of $X$ that is isomorphic to $A'$. This yields a factorization of $k$ through $\eta _X$ as the composite
\[X\xrightarrow{\eta _X} DX\xrightarrow{g} X/R\xrightarrow{\cong } A'\xrightarrow{h} A\]
is equal to $k$.
\end{proof}
\noindent If we fix a simplicial set $X$, then we can consider degreewise surjective maps $k:X\to A$ whose targets are non-singular. When factored through $\eta _X$, the resulting unique maps $\bar{k} :DX\to A$ are degreewise surjective. In this sense, desingularization is the least drastic way of forming a non-singular quotient from a (possibly singular) simplicial set.

\subsection{Functorality of Desingularization and (co)limits in (the category of) non-singular simplicial sets}

It is possible to define $D$ on morphisms in a straightforward way. Then one realizes that the construction is functorial and that $\eta _X$ is natural as a map $X\to UDX$. If $A$ is non-singular, then $\eta _{UA}$ is an isomorphism. This is observed by factoring the identity $UA\to UA$ through $\eta _{UA}$ by means of \cref{lem:desing_unique_factorization_through_unit}. As $U$ is a full embedding, the latter fact suggests the formulation of \cref{lem:non-singular_reflective_subcategory} below.

A full subcategory of some category is a \textbf{reflective subcategory} if the inclusion admits a left adjoint. The terminology is not quite standard as the fullness assumption is omitted by some, for example in \cite[§IV.3]{ML98} and \cite[p.~1306]{AR15}. As announced, we have the following result \cite[Rem.~2.2.12]{WJR13}.
\begin{lemma}\label{lem:non-singular_reflective_subcategory}
The category of non-singular simplicial sets is a reflective subcategory of $sSet$.
\end{lemma}
\begin{proof}
We will prove the lemma by establishing the natural map $\eta _X$ as the unit of a pair $(\eta ,\epsilon )$ consisting of a unit and a counit $\epsilon$.

Let $f:A\to B$ be a morphism in $nsSet$. Consider the diagram
\begin{displaymath}
 \xymatrix{
 & UD(UB) \\
 && UD(UA) \ar[lu]_{UD(Uf)} \ar[ddr]^{U(\epsilon _A)=\eta ^{-1}_{UA}} \\
 UB \ar[uur]_\cong ^{\eta _{UB}} \\
 & UA \ar[lu]^{Uf} \ar[uur]_\cong ^{\eta _{UA} } \ar[rr]_{id} && UA
 }
\end{displaymath}
in which the inverse $\eta ^{-1}_{UA}$ appears. As $U$ is full, the latter is equal to $U(\epsilon _A)$ for some morphism $\epsilon _A:DU\, A\to A$ of $nsSet$. It is evident from the outer part of the diagram that $\epsilon _A$ is natural in $A$.

The triangle at the right hand side, which defines $\epsilon _A$, is the first half of the compatibility criterion that a unit and a counit must satisfy. The commutative square
\begin{displaymath}
\xymatrix{
 X \ar[d]_{\eta _X} \ar[rr]_{\eta _X} && UDX \ar[d]^{UD(\eta _X)} \\
 UDX \ar[rr]_{\eta _{UDX}}^\cong && UD(UDX)
}
 \end{displaymath}
shows that
\[\eta _{UDX}=UD(\eta _X)\]
for every simplicial set $X$. If we combine this with the definition of $\epsilon _{DX}$, then we get the commutative triangle-shaped diagram
\begin{displaymath}
\xymatrix{
& D(UDX)=DU(DX) \ar[dr]^(.6){\epsilon _{DX}} \\
DX \ar[ur]^(.4){D(\eta _X)} \ar[rr]_{id} && DX
}
\end{displaymath}
which is the second half of the compatibility criterion. This concludes the verification that $D$ is left adjoint to the inclusion $U$.
\end{proof}
\noindent The implication of \cref{lem:non-singular_reflective_subcategory} is that it has a strong bearing on the formation of (co)limits of diagrams in $nsSet$, as we now explain.

A diagram in a reflective subcategory has a limit if it has a limit when considered a diagram in the surrounding category. In that case, the limit is inherited by the subcategory. See for example \cite[p.~92]{ML98} or \cite[p.~1306]{AR15}. Consequently, $nsSet$ is complete as $sSet$ is.

The colimit in a reflective subcategory can be formed by taking it in the surrounding category, if it exists there, and then applying the reflector. As the counit of an adjunction is an isomorphism whenever the right adjoint is fully faithful \cite[§IV.3~Thm.1]{ML98}, we obtain a colimit of the original diagram. The reflector is in our case desingularization. Thus $nsSet$ is cocomplete because $sSet$ is cocomplete, although this way of computing a colimit in $nsSet$ requires knowledge of desingularization.

For later reference we record the following consequence of \cref{lem:non-singular_reflective_subcategory}.
\begin{corollary}\label{cor:nsSet_bicomplete}
The category $nsSet$ of non-singular simplicial sets is bicomplete.
\end{corollary}

%% file: sections/methods.tex
\section{Calculational methods}
\label{sec:calculations}

As far as we know, the only explicit description of $DX$ that is present in the literature is that of \cref{def:desingularization}. It has the advantage that we easily obtain \cref{lem:desing_unique_factorization_through_unit}. However, the description is otherwise rather difficult to work with. Consequently, we would like some tools to aid in calculation. In this section, we will make a couple of observations that are actually enough to perform a few simple, yet interesting desingularizations.

It is maybe in order that the following near-trivial example be mentioned first.
\begin{example}\label{ex:simplest_non-trivial_desingularization}
Consider a simplicial set $X$ whose set $X_0$ of $0$-simplices is a singleton. It follows immediately from the definition of the term non-singular that any simplex of $A=DX$ is degenerate if its degree is $1$ or higher. If $a$ is a simplex of $A$, then we can write it uniquely as a degeneration
\[a=a^\sharp a^\flat\]
of a non-degenerate simplex $a^\sharp$, by the Eilenberg-Zilber lemma. As we have just argued, the only non-degenerate simplex is the single $0$-simplex, so $a^\sharp$ is that one. If $a$ and $b$ have the same dimension, $n$ say, then
\[a^\flat =a^\flat\]
as there is only one operator $[n]\to [0]$. This proves that the set $A_n$ of $n$-simplices is a singleton, implying that the unique map
\[DX\xrightarrow{\cong } \Delta [0]\]
is an isomorphism.
\end{example}
\noindent Arguably, \cref{ex:simplest_non-trivial_desingularization} is the simplest non-trivial example.

Let $X$ be a simplicial set. Towards the goal of making it non-singular we would need to force any non-embedded non-degenerate simplex into becoming degenerate. Suppose $x\in X^\sharp _{n_x}$. The simplex $x$ is embedded if and only if its vertices are pairwise distinct. If it is not embedded, then we would like to make it degenerate according to any pairwise equalities between its vertices. To achieve this we begin by defining a reflexive, symmetric binary relation $\sim$ on
\[O([n_x])=\{ 0,\dots ,n_x\}\]
by letting
\[i\sim j\Leftrightarrow x\varepsilon _i=x\varepsilon _j.\]
Next, we can define a reflexive binary relation $\approx$ on $O([n_x])$ by letting $i\approx k$ if and only if there is a $j$ such that $i\leq k\leq j$ in the total order on $[n_x]$ and such that $i\sim j$. If $i\sim j$ and $i\leq j$, then $i\approx j$. This means that $\sim$ is contained in the equivalence relation $\simeq$ on $O([n_x])$ that is generated by $\approx$.

Crucially, the equivalence relation $\simeq$ has the property described in the following result.
\begin{lemma}\label{lem:equivalence_relation_gives_rise_to_enforcer}
The equivalence relation $\simeq$ on $O([n_x])$ that is generated by $\approx$ has the property that if $i\simeq j$ and if $i\leq k\leq j$ in the total order on $[n_x]$, then $i\simeq k$.
\end{lemma}
\begin{proof}
Assume that $i\simeq j$ and that $i\leq k\leq j$ in the total order on $[n_x]$. Consider the non-trivial case $i<j$.

In the special case when $i\approx j$, there is a $j'$ such that $i\leq j\leq j'$ and $i\sim j'$. As $j\leq j'$ and $i\leq k\leq j$ we get that $i\leq k\leq j'$. Because $i\sim j'$ we then get that $i\approx k$ from the definition of this binary relation, which implies $i\simeq k$.

If it is not true that $i\approx j$, then we still have elements
\[i_0,\dots ,i_q\in O([n_x])\]
for some $q>1$ such that
\begin{displaymath}
\begin{array}{rcl}
i & = & i_0 \\
i_q & = & j
\end{array}
\end{displaymath}
and
\begin{displaymath}
\begin{array}{rclcrcl}
i_0 & \approx & i_1 & \textrm{ or } & i_1 & \approx & i_0 \\
&&& \dots \\
i_{q-1} & \approx & i_q & \textrm{ or } & i_q & \approx & i_{q-1}
\end{array}
\end{displaymath}
There is some $p<q$ such that $i_p\leq k\leq i_{p+1}$, in the case when $i_p\approx i_{p+1}$, or that $i_p\geq k\geq i_{p+1}$, in the case when $i_{p+1}\approx i_p$. Thus $i\simeq k$.
\end{proof}
\noindent An immediate consequence of \cref{lem:equivalence_relation_gives_rise_to_enforcer} is that the set $O([n_x])/\simeq$ has a canonical total order $\leq$ that the canonical function
\[O([n_x])\to O([n_x])/\simeq\]
respects.

If $m_x+1$ is the cardinality of the set $O([n_x])/\simeq$, then the canonical identification
\[(O([n_x])/\simeq ,\leq )\xrightarrow{\cong } [m_x]\]
suggested above gives rise to the method of enforcing the rules of glueing in $nsSet$.
\begin{definition}\label{def:enforcer}
Let $x$ be a non-degenerate simplex of some simplicial set. Define $\rho _x$ as the composite
\[[n_x]\to (O([n_x])/\simeq ,\leq )\xrightarrow{\cong } [m_x].\]
Let the degeneracy operator $\rho _x$ be known as the \textbf{enforcer of $x$}.
\end{definition}
\noindent In general, the degeneracy operators whose source is $[n_x]$ correspond to equivalence relations on the set $O([n_x])$ that satisfy precisely the condition from \cref{lem:equivalence_relation_gives_rise_to_enforcer}.

The name of $\rho _x$ is meant to signify that it has a role in making sure that the result of desingularizing $X$ is a simplicial set that obeys the rules of glueing in the category $nsSet$. These are stricter than the rules in the category $sSet$. By construction, the enforcer deals with any equalities between vertices of $x$, but in the least drastic manner. It is proper if and only if $x$ is not embedded.
\begin{proposition}\label{prop:role_of_enforcers}
Let $J\subseteq X^\sharp$ be some set of non-degenerate simplices. There is a canonical map
\[\bigsqcup _{j\in J}\Delta [m_j]\to UDX\]
such that the square
\begin{equation}
\label{eq:first_diagram_proof_prop_role_of_enforcers}
\begin{gathered}
\xymatrix{
\bigsqcup _{j\in J}\Delta [n_j] \ar[rr]^{\sqcup _{j\in J}(\rho _j)} \ar[d]_{\vee _{j\in J}(\bar{\jmath } )} && \bigsqcup _{j\in J}\Delta [m_j] \ar[d] \\ 
X \ar[rr]_{\eta _X} && UDX
}
\end{gathered}
\end{equation}
commutes.
\end{proposition}
\begin{proof}
First, note that if $x\in X^\sharp _n$, then the composite
\[\Delta [n_x]\xrightarrow{\bar{x} } X\xrightarrow{\eta _X} UDX\]
factors through $N\rho _x$. One realizes this by considering the image $z$ of $x$ under $\eta _X$. It is uniquely a
degeneracy of a non-degenerate simplex. We get the diagram
\begin{displaymath}
 \xymatrix{
 \Delta [m_x] \ar@{-->}[dr] \\
 & \Delta [m] \ar[dr]^{\overline{z^\sharp } } \\
 \Delta [n_x] \ar[uu]^{\rho _x} \ar[ur]^{z^\flat } \ar[dr]_{\bar{x} } \ar[rr]^{\bar{z} } && UDX \\
 & X \ar[ur]_{\eta _X}
 }
\end{displaymath}
in which $Nz^\flat$ factors uniquely through $N\rho _x$. The explanation for the latter factorization is as follows.

That there is at most one factorization comes from the fact that the nerve $N$ is fully faithful and that $\rho _x$ is epic in $Cat$. That there is a factorization follows from the observation that $z^\flat (i)=z^\flat (i')$ whenever $\rho _x(i)=\rho _x(i')$, as we now argue.

First, suppose $i\sim i'$, meaning $x\varepsilon _i=x\varepsilon _{i'}$. As $\eta _X$ is a simplicial map it follows that
\[z^\sharp \varepsilon _{z^\flat (i)}=z^\sharp z^\flat \varepsilon _i=z\varepsilon _i=z\varepsilon _{i'}=z^\sharp z^\flat \varepsilon _{i'}=z^\sharp \varepsilon _{z^\flat (i')}.\]
The simplicial set $UDX$ is non-singular, so $z^\sharp$ is embedded. Hence,
\[z^\flat (i)=z^\flat (i').\]
Next, as $z^\flat$ is order-preserving we have that $z^\flat (i)=z^\flat (k)$ for each $k$ with $i\leq k\leq i'$. As a consequence, $z^\flat (i)=z^\flat (k)$ whenever $i\approx k$.

In turn, we get that the the equivalence relation $\simeq$ on $O([n_x])$, which corresponds to $\rho _x$, is contained in the equivalence relation that corresponds to $z^\flat$. Thus we obtain a canonical degeneration $w_x$ of $z^\sharp$ such that the square
\begin{displaymath}
\xymatrix{
\Delta [n_x] \ar[d]_{\bar{x} } \ar[r]^{\rho _x} & \Delta [m_x] \ar[d]^{\bar{w} _x} \\
X \ar[r]_{\eta _X} & UDX
}
\end{displaymath}
commutes.

The composites
\[\Delta [n_j]\xrightarrow{\bar{\jmath } } X\xrightarrow{\eta _X} DX,\]
$j\in J$, give rise to a canonical map $\sqcup _{j\in J}\Delta [n_j]\to DX$. The latter can be factored in two different ways due to (\ref{eq:first_diagram_proof_lem_pushout_along_enforcers_intermediate_desingularization}).

The diagram illustrated by
\begin{displaymath}
\xymatrix{
\bigsqcup _{j\in J}\Delta [n_j] \ar@{-->}[r] & \bigsqcup _{j\in J}\Delta [m_j] \ar@{-->}[r] & UDX \\
\Delta [n_x] \ar[u] \ar[r]_{\rho _x} & \Delta [m_x] \ar[u] \ar[ur]_{\bar{w} _x}
}
\end{displaymath}
provides the first of the factorizations that we have in mind and the diagram
\begin{displaymath}
\xymatrix{
\Delta [n_x] \ar[d] \ar[dr]^{\bar{x} } \\
\bigsqcup _{j\in J}\Delta [n_j] \ar@{-->}[r] & X \ar[r]^{\eta _X} & DX
}
\end{displaymath}
provides the second. The promised commutative square consists of precisely these two factorizations.
\end{proof}
\begin{lemma}\label{lem:pushout_along_enforcers_intermediate_desingularization}
Let $X$ be a simplicial set and let $J\subseteq X^\sharp$ be some set of non-degenerate simplices. Consider the cocartesian square
\begin{displaymath}
\xymatrix{
\bigsqcup _{j\in J}\Delta [n_{j}] \ar[rr]^{\sqcup _{j\in J}(\rho _j)} \ar[d]_{\vee _{j\in J}(\bar{\jmath } )} && \bigsqcup _{j\in J}\Delta [m_j] \ar[d] \\
X \ar[rr] && Y
}
\end{displaymath}
in $sSet$. The unit $\eta _X$ factors through the canonical degreewise surjective map $X\to Y$. If $Y$ is non-singular, then the map $Y\to UDX$ of the factorization is an isomorphism.
\end{lemma}
\begin{proof}
As a result of \cref{prop:role_of_enforcers}, the solid diagram
\begin{displaymath}
\xymatrix{
\bigsqcup _{j\in J}\Delta [n_j] \ar[rr]^{\sqcup _{j\in J}(\rho _j)} \ar[d]_{\vee _{j\in J}(\bar{\jmath } )} && \bigsqcup _{j\in J}\Delta [m_j] \ar[d] \ar@/^1pc/[ddr] \\ 
X \ar@/_1pc/[drrr]_{\eta _X} \ar[rr] && Y \ar@{-->}[dr] \\
&&& UDX
}
\end{displaymath}
commutes. Thus a canonical map $Y\to UDX$ arises. It is degreewise surjective as $\eta _X$ is degreewise surjective.

Suppose $Y$ non-singular. We will argue that $Y\to UDX$ is even degreewise injective in this case and that it is thus an isomorphism. We get the commutative diagram
\begin{equation}
\label{eq:first_diagram_proof_lem_pushout_along_enforcers_intermediate_desingularization}
\begin{gathered}
\xymatrix{
& UDX \\
X \ar[dr]_{\eta _X} \ar[ur]^{\eta _X} \ar[rr] && Y \ar[lu] \\
& UDX \ar@{-->}[ur]
}
\end{gathered}
\end{equation}
in which the upper triangle comes from the pushout above and in which the lower triangle comes from $Y$ being non-singular. Hence, the composite
\[UDX\to Y\to UDX\]
is the identity as $\eta _X$ is epic in $sSet$. This implies that $UDX\to Y$ is degreewise injective.

The canonical map $X\to Y$ that comes with the pushout $Y$ is degreewise surjective as it is a cobase change of a degreewise surjective map. Consequently, we can conclude from (\ref{eq:first_diagram_proof_lem_pushout_along_enforcers_intermediate_desingularization}) that the map $UDX\to Y$ is degreewise surjective. This implies that $Y\to UDX$ is degreewise injective in this case.
\end{proof}
\noindent \cref{lem:pushout_along_enforcers_intermediate_desingularization} confirms the intuition that taking the pushout along enforcers is never too drastic.

%% file: sections/calc.tex
\section{A few calculations}
\label{sec:examples}

What happens if one desingularizes, say the result of collapsing the second face of a standard $2$-simplex, as in \cref{fig:ch2_Desing_2simpl_collapsed_second_face}? The dashed line segment is meant to indicate that the second face has been collapsed. The dotted lines are meant to illustrate the identifications that arise as a result of the desingularization. The next example is a slight generalization in that it replaces $2$ with $n$ and $\delta _2$ with $\delta _n\cdots \delta _{k+1}$ for some $k$ that replaces $1$. We will use the notion of enforcer from \cref{def:enforcer}.
\begin{example}\label{ex:second_simplest_non-trivial_desingularization}
Let $\mu :[k]\to [n]$ be the face operator defined by
\[\mu =\delta _n\cdots \delta _{k+1}.\]
Consider the cocartesian square
\begin{equation}
\label{eq:first_diagram_ex_second_simplest_non-trivial_desingularization}
\begin{gathered}
\xymatrix{
\Delta [k] \ar[d]_{\mu } \ar[r] & \Delta [0] \ar[d] \\
\Delta [n] \ar[r]_{\bar{x} } & X
}
\end{gathered}
\end{equation}
in $sSet$, in the non-trivial case $0<k<n$. The non-degenerate simplex $x$ is then not embedded. We will argue that
\[DX\cong \Delta [n-k]\]
by use of the decomposition of $X$ as the pushout above.

The enforcer
\[\rho _x=\sigma _0\cdots \sigma _{k-1}\]
of $x$ fits into the commutative solid diagram
\begin{equation}
\label{eq:second_diagram_ex_second_simplest_non-trivial_desingularization}
\begin{gathered}
\xymatrix{
\Delta [k] \ar[d]_{\mu } \ar[r] & \Delta [0] \ar[d] \ar@/^1pc/[ddr]^{\varepsilon _0} \\
\Delta [n] \ar[r]_{\bar{x} } \ar@/_2pc/[drr]_{\rho _x} & X \ar@{-->}[dr] \\
&& \Delta [n-k]
}
\end{gathered}
\end{equation}
in $sSet$, which gives rise to a canonical dashed map $X\to \Delta [n-k]$. Next, consider the diagram
\begin{equation}
\label{eq:third_diagram_ex_second_simplest_non-trivial_desingularization}
\begin{gathered}
\xymatrix{
\Delta [n] \ar[d]_{\bar{x} } \ar[r]^(.40){\rho _x} & \Delta [n-k] \ar[d] \ar@/^1pc/[ddr]^{id} \\
X \ar@/_1pc/[drr] \ar[r] & Y \ar@{-->}[dr] \\
&& \Delta [n-k]
}
\end{gathered}
\end{equation}
in $sSet$ where $X\to \Delta [n-k]$ is the map from (\ref{eq:second_diagram_ex_second_simplest_non-trivial_desingularization}). In (\ref{eq:third_diagram_ex_second_simplest_non-trivial_desingularization}), the simplicial set $Y$ is the pushout
\[Y=X\sqcup _{\Delta [n]}\Delta [n-k].\]
From the triangle on the right it follows that $\Delta [n-k]\to Y$ is degreewise injective and that $Y\to \Delta [n-k]$ is degreewise surjective.

Meanwhile, the map $\Delta [n-k]\to Y$ is a cobase change of the map $\bar{x}$, which is itself a cobase change of the degreewise surjective map $\Delta [k]\to \Delta [0]$, as is seen from (\ref{eq:first_diagram_ex_second_simplest_non-trivial_desingularization}). Hence, the map $\Delta [n-k]\to Y$ is degreewise surjective. It follows that $Y\xrightarrow{\cong } \Delta [n-k]$ is an isomorphism. Thus $Y$ is seen to be non-singular. From \cref{lem:pushout_along_enforcers_intermediate_desingularization}, we get that
\[DX\cong Y\cong \Delta [n-k],\]
which was our claim.
\end{example}
\noindent The computation of $DX$ in \cref{ex:second_simplest_non-trivial_desingularization} was particularly easy because of the unusual decomposition of $X$ which in turn arose partly from the fact that $X$ was generated by a single non-degenerate simplex.

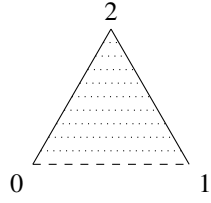
\begin{figure}
\centering
\begin{tikzpicture}[scale=0.4]
\coordinate (2) at (90:3cm);
\coordinate (0) at (210:3cm);
\coordinate (1) at (-30:3cm);

\node [above] at (2) {2};
\node [below left] at (0) {0};
\node [below right] at (1) {1};

\draw (0.north east)--(2.south) coordinate[midway](02);
\draw [dashed] (0.north east)--(1.north west) coordinate[midway](01);
\draw (1.north west)--(2.south) coordinate[midway](12);

\foreach \s in {1,...,9}
	\draw [dotted] (barycentric cs:0=\s/10,2=1-\s/10)--(barycentric cs:1=\s/10,2=1-\s/10);

\end{tikzpicture}
\caption{Desingularizing the standard $2$-simplex whose second face has been collapsed.}
\label{fig:ch2_Desing_2simpl_collapsed_second_face}
\end{figure}

Let us consider a few models of spheres. The first ones have desingularizations that can be calculated simply by an inspection and ad hoc arguments.
\begin{example}\label{ex:desingularization_model_n-sphere}
Consider the cocartesian square
\begin{displaymath}
\xymatrix{
\partial \Delta [n] \ar[d] \ar[r] & \Delta [0] \ar[d] \\
\Delta [n] \ar[r]_(.35){\bar{x} } & \Delta [n]/\partial \Delta [n]
}
\end{displaymath}
in $sSet$. The non-degenerate simplex $x$ is not embedded if $n>0$. In the case when $n=0$, we get
\[\Delta [0]/\partial \Delta [0]\cong \Delta [0]\sqcup \Delta [0],\]
which is non-singular. In other words, desingularization has no effect on $\Delta [0]/\partial \Delta [0]$. Else if $n>0$, then we can apply \cref{ex:simplest_non-trivial_desingularization} to obtain
\[D(\Delta [n]/\partial \Delta [n])\cong \Delta [0].\]
The latter calculation shows that desingularization has homotopically destructive tendencies.
\end{example}
\noindent We record the results from \cref{ex:desingularization_model_n-sphere} in \cref{tab:Desing_spherical_mod} below, which is explained shortly.

What if we subdivide the model $\Delta [n]/\partial \Delta [n]$ of the $n$-sphere before applying desingularization? Let $Sd$ denote the Kan subdivision. See \cite[Def.~2.2.7]{WJR13} or \cite[p.~148]{FP90} for a definition. The Kan subdivision is the left Kan extension of barycentric subdivision along the Yoneda embedding \cite[p.~37]{WJR13}, so to get a mental picture of its effect one can think of barycentric subdivision. There are illustrations of desingularizations of Kan subdivisions in \cref{fig:ch2_Desing_singlysubd_2simpl} and \cref{fig:ch2_Desing_doublysubd_2simpl}. Note that $Sd$ preserves degreewise injective maps \cite[Cor.~4.2.9]{FP90} and that it has a right adjoint \cite[Prop.~4.2.10]{FP90}. In particular, the Kan subdivision preserves attachings.

At this point, we introduce the \textbf{Barratt nerve} \cite[Def.~2.2.3]{WJR13}
\[BX=N(X^\sharp )\]
of a simplicial set $X$ for comparison with $Sd\, X$. Here, we let the set $X^\sharp$ of non-degenerate simplices have the partial order $\leq$ defined by letting $y\leq x$ if $y$ is a face of $x$. We think of a partially ordered set, poset for short, $(P,\leq )$ as a small category by letting the elements of $P$ be the objects and we let there be a morphism $p\to p'$ if $p\leq p'$. One can interpret $B$ as an endofunctor of simplicial sets, although its image is in the full subcategory $nsSet$. Indeed, the Barratt nerve $BX$, of every simplicial set $X$, is the simplicial set associated with an ordered simplicial complex. There is \cite[p.~37]{WJR13} a natural degreewise surjective \cite[Lem.~2.2.10]{WJR13} map
\[b_X:Sd\, X\to BX,\]
which is an isomorphism if and only if $X$ is non-singular \cite[Lem.~2.2.11]{WJR13}.

\begin{table}
\centering
\begin{tikzpicture}
    \matrix(dict)[matrix of nodes,
    nodes={align=center,text width=8em},
    row 1/.style={anchor=south},
    column 1/.style={nodes={text width=4.5em,align=right}}
    ]{
        $DSd^k(X)$ & $k=0$ & $k=1$ & $k=2$ \\
        $n=0$ & $\Delta [0]\sqcup \Delta [0]$ & $\Delta [0]\sqcup \Delta [0]$ & $\Delta [0]\sqcup \Delta [0]$ \\
        $n=1$ & $\Delta [0]$ & $\Delta [1]\sqcup _{\partial \Delta [1]}\Delta [1]$ & $A\sqcup _{\partial A}A$ \\
        $n=2$ & $\Delta [0]$ & $\Delta [1]$ & $S(12-gon)$ \\
    };
    \draw(dict-1-1.south west)--(dict-1-4.south east);
    \draw(dict-1-1.north east)--(dict-4-1.south east);
\end{tikzpicture}
\caption{Desingularizations of models of certain spheres. Here, we denote $X=\Delta [n]/\partial \Delta [n]$, $A=Sd(\Delta [1])$ and $\partial A=Sd(\partial \Delta [1])$.}
\label{tab:Desing_spherical_mod}
\end{table}

Let $Sd^k$ denote the Kan subdivision applied $k$ times, for each integer $k\geq 0$. We consider $X=Sd^k(\Delta [n]/\partial \Delta [n])$ for $0\leq n\leq 2$ and $0\leq k\leq 2$. As we obtain desingularizations of these simplicial sets, we record the results in \cref{tab:Desing_spherical_mod}. \cref{ex:desingularization_model_n-sphere} took care of the case when $k=0$. Furthermore, we have the following calculations.
\begin{example}\label{ex:all_subdivisions_sphere_dim_zero_coincidence_dim_one}
For every $k\geq 0$, we have
\[Sd^k(\Delta [0]/\partial \Delta [0])\cong \Delta [0]\sqcup \Delta [0].\]
So too, for $k=1$ and $k=2$. Applying desingularization has no effect as $Sd^k(\Delta [0]/\partial \Delta [0])$ is already non-singular. By a coincidence, the simplicial set $Sd(\Delta [1]/\partial \Delta [1])$ is also non-singular, as we explain next.

The commutative square
\begin{displaymath}
\xymatrix{
[0] \ar[d]_{\varepsilon _1} \ar[rr]^{\varepsilon _1} && [1] \ar[d]^{0\mapsto \varepsilon _0,\; 1\mapsto \iota } \\
[1] \ar[rr]_{0\mapsto \varepsilon _1,\; 1\mapsto \iota } && \Delta [1]^\sharp
}
\end{displaymath}
where $\iota$ denotes the identity, gives rise to a canonical map
\[\Delta [1]\sqcup _{\Delta [0]}\Delta [1]\xrightarrow{\cong } B(\Delta [1])\]
that is an isomorphism. Inverting it and forming the composite
\[Sd(\Delta [1])\xrightarrow{b_{\Delta [1]}} B(\Delta [1])\xrightarrow{\cong } \Delta [1]\sqcup _{\Delta [0]}\Delta [1]\]
which is in turn precomposed with the canonical map
\[\Delta [1]\sqcup \Delta [1]\to \Delta [1]\sqcup _{\partial \Delta [1]}\Delta [1]\]
yields the solid diagram
\begin{displaymath}
\xymatrix@=0.9em{
Sd(\partial \Delta [1]) \ar[d] \ar[r] & Sd(\Delta [0]) \ar[d] \ar@/^2pc/[ddr] \\
Sd(\Delta [1]) \ar@/_2pc/[drr] \ar[r] & Sd(\Delta [1]/\partial \Delta [1]) \ar@{-->}[dr]^\cong \\
&& \Delta [1]\sqcup _{\partial \Delta [1]}\Delta [1]
}
\end{displaymath}
that commutes, thus giving rise to a canonical dashed map that is in fact an isomorpism. Then the desingularization is trivially
\[DSd(\Delta [1]/\partial \Delta [1])\cong \Delta [1]\sqcup _{\partial \Delta [1]}\Delta [1],\]
which is also recorded in \cref{tab:Desing_spherical_mod}.
\end{example}
\noindent We resume with slightly more complicated examples.

By \cref{lem:desing_unique_factorization_through_unit}, we obtain a map $t_X:DSd\, X\to BX$. Because $\eta _{Sd\, X}$ and $b_X$ are natural, because $\eta _{Sd\, X}$ is degreewise surjective and because the target of $b_X$ is non-singular, the map $t_X$ can be interpreted as a natural map between functors $sSet\to nsSet$ when we corestrict $B$ to $nsSet$.

We will prove the following result.
\begin{proposition}\label{prop:double_subdivision_sphere_low_dimension}
The map
\[DSd^2(\Delta [n]/\partial \Delta [n])\xrightarrow{t_{Sd(\Delta [n]/\partial \Delta [n])}} BSd(\Delta [n]/\partial \Delta [n])\]
is an isomorphism for $0\leq n\leq 2$.
\end{proposition}
\noindent For the proof of \cref{prop:double_subdivision_sphere_low_dimension}, note we have already taken care of the case when $n=0$ in \cref{ex:all_subdivisions_sphere_dim_zero_coincidence_dim_one}. The case when $n=1$ follows from \cref{ex:double_subdivision_sphere_dim_one} below.
\begin{example}\label{ex:double_subdivision_sphere_dim_one}
By \cref{ex:all_subdivisions_sphere_dim_zero_coincidence_dim_one}, the simplicial set $Sd(\Delta [1]/\partial \Delta [1])$ is non-singular. Therefore the map $b_{Sd(\Delta [1]/\partial \Delta [1])}$ is an isomorphism, which implies that the map
\[t_{Sd(\Delta [1]/\partial \Delta [1])}:DSd^2(\Delta [1]/\partial \Delta [1])\xrightarrow{\cong } BSd(\Delta [1]/\partial \Delta [1])\]
is an isomorphism.
\end{example}
\noindent To prove \cref{prop:double_subdivision_sphere_low_dimension}, it remains to consider the case when $n=2$.

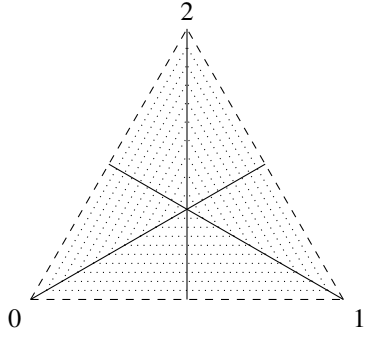
\begin{figure}
\centering
\begin{tikzpicture}[scale=0.8]
\coordinate (2) at (90:3cm);
\coordinate (0) at (210:3cm);
\coordinate (1) at (-30:3cm);

\node [above] at (2) {2};
\node [below left] at (0) {0};
\node [below right] at (1) {1};

\draw [dashed] (0.north east)--(2.south) coordinate[midway](02);
\draw [dashed] (0.north east)--(1.north west) coordinate[midway](01);
\draw [dashed] (1.north west)--(2.south) coordinate[midway](12);

\coordinate (012) at (barycentric cs:0=1,1=1,2=1);

\foreach \x in {(0),(01),(1),(12),(2),(02)}
	\draw (012)--\x;

\foreach \x in {0,1}
	\foreach \y in {1,2}
		\foreach \s in {1,...,9}
			\draw [dotted] (barycentric cs:\x=\s/10,012=1-\s/10)--(barycentric cs:\y=\s/10,012=1-\s/10);

\end{tikzpicture}
\caption{Desingularizing the Kan subdivision of the $2$-simplex with collapsed boundary.}
\label{fig:ch2_Desing_singlysubd_2simpl}
\end{figure}

Before we prove \cref{prop:double_subdivision_sphere_low_dimension} in the case when $n=2$, we contemplate how to desingularize $Sd(\Delta [2]/\partial \Delta [2])$, which is a similar task, although slightly easier. We make use of the notion of enforcer from \cref{def:enforcer}.
\begin{example}\label{ex:desingularization_model_n-sphere_subdivided}
Consider the cocartesian square
\begin{equation}
\label{eq:first_diagram_ex_desingularization_model_n-sphere_subdivided}
\begin{gathered}
\xymatrix{
Sd(\partial \Delta [2]) \ar[d] \ar[r] & Sd(\Delta [0]) \ar[d] \\
Sd(\Delta [2]) \ar[r] & Sd\, X
}
\end{gathered}
\end{equation}
in $sSet$, where we have written $X=\Delta [2]/\partial \Delta [2]$ for brevity. We will prove that
\begin{equation}\label{eq:first_equation_desingularization_model_n-sphere_subdivided}
DSd\, X\cong \Delta [1].
\end{equation}
In \cref{fig:ch2_Desing_singlysubd_2simpl}, we illustrate the effect of desingularizing $Sd\, X$. This illustration indicates the idea of the proof and is helpful in bookkeeping. The dashed line segments that are part of the boundary are meant to indicate that the boundary has been collapsed in order to form $\Delta [2]/\partial \Delta [2]$. The dotted line segments are meant to illustrate how identifications arise when desingularizing.

The simplicial set $Sd\, X$ is generated by six (non-degenerate) $2$-simplices as $Sd(\Delta [2])$ is generated by six non-degenerate $2$-simplices and as
\[Sd(\Delta [2])\to Sd\, X\]
is degreewise surjective. We will name these six generators. Let the simplex $y_1$ be the image under
\[B(\Delta [2])\cong Sd(\Delta [2])\to Sd\, X\]
of the simplex $\{0<01<012\}$. Furthermore, let $y_2$ be the image of the next non-degenerate $2$-simplex $\{1<01<012\}$ as we move counterclockwise in \cref{fig:ch2_Desing_singlysubd_2simpl} and so on up to and including $j=6$. Thus the set
\[\{y_j\} _{j\in J},\, J=\{ 1,\dots ,6\}\]
generates $Sd\, X$.

The simplicial set $Sd(\Delta [2])$ has seven $0$-simplices that correspond to the seven elements of $\Delta [2]^\sharp$. The six $0$-simplices on the boundary $Sd(\partial \Delta [2])$ are identified with each other when $Sd\, X$ is formed from $Sd(\Delta [2])$. However, the $0$-simplex $012$ is not identified with these. Write $z_j=\eta _{Sd\, X}(y_j)$ for each $j\in J$. Each of the $2$-simplices $y_j$, $j\in J$, is such that the vertices $y_j\varepsilon _0$ and $y_j\varepsilon _1$ are on the boundary and that $y_j\varepsilon _2$ is equal to $012$. Thus we see that each of the simplices $y_j$, $j\in J$, has the elementary degeneracy operator
\[\rho _{y_j}=\sigma _0\]
as its enforcer. Let $\rho$ denote this common enforcer.

For each $j\in J$, write $z_j=\eta _{Sd\, X}(y_j)$. From \cref{prop:role_of_enforcers}, we have the commutative square
\begin{equation}
\label{eq:second_diagram_ex_desingularization_model_n-sphere_subdivided}
\begin{gathered}
\xymatrix{
\bigsqcup _{j\in J}\Delta [2] \ar[d]_{\vee _{j\in J}(\bar{y} _j)} \ar[rr]^{\sqcup _{j\in J}(\rho )} && \bigsqcup _{j\in J}\Delta [1] \ar[d]^{\vee _{j\in J}(\bar{w} _j)} \\
Sd\, X \ar[rr]_{\eta _{Sd\, X}} && UDSd\, X
}
\end{gathered}
\end{equation}
in $sSet$, where $w_j$ is the canonical degeneracy of the non-degenerate part of $z_j$, $j\in J$.  In this case, the simplices $w_j$, $j\in J$, are embedded and therefore non-degenerate. This way we see how the simplices $z_j$, $j\in J$, are degenerate.

Because the simplices $z_j$, $j\in J$, are all degenerate it follows that $DSd\, X$ is generated by the images under $\eta _{Sd\, X}$ of the six embedded $1$-simplices of $Sd\, X$. We will argue that all of these images are equal.

Pick a $j\in J$. Two of the six embedded $1$-simplices of $Sd\, X$ are the faces $y_j\delta _1$ and $y_j\delta _0$ of $y_j$. Because $\delta _1$ and $\delta _0$ are both sections of $\rho$, we get that
\begin{displaymath}
\begin{array}{rclclcl}
z_j\delta _1 & = & (w_j\rho )\delta _1 & = & w_j(\rho \delta _1) & = & w_j \\
z_j\delta _0 & = & (w_j\rho )\delta _0 & = & w_j(\rho \delta _0) & = & w_j.
\end{array}
\end{displaymath}
Thus it follows that the image under $\eta _{Sd\, X}$ of each of the faces $y_j\delta _1$ and $y_j\delta _0$ is equal to $w_j$. Let us express this with $y_j\delta _1\sim y_j\delta _0$ for each $j\in J$.

By moving counterclockwise in \cref{fig:ch2_Desing_singlysubd_2simpl}, we get that
\begin{displaymath}
\begin{array}{rclcl}
y_1\delta _0 & = & y_2\delta _0 & \sim & y_2\delta _1 \\
y_2\delta _1 & = & y_3\delta _1 & \sim & y_3\delta _0 \\
y_3\delta _0 & = & y_4\delta _0 & \sim & y_4\delta _1 \\
y_4\delta _1 & = & y_5\delta _1 & \sim & y_5\delta _0 \\
y_5\delta _0 & = & y_6\delta _0 & \sim & y_6\delta _1.
\end{array}
\end{displaymath}
This shows that
\[w_1=w_2=\cdots =w_6,\]
implying that (\ref{eq:first_equation_desingularization_model_n-sphere_subdivided}) holds.
\end{example}
\noindent To complete \cref{tab:Desing_spherical_mod}, the only remaining case is when $k=2$ and $n=2$.

Note that the functor $BSd$ replaces a simplicial set with an ordered simplicial complex of the same homotopy type \cite[Ex.~3--8, pp.~219--220]{FP90}. To conjecture the homotopical content of the claim of \cref{prop:double_subdivision_sphere_low_dimension} one uses the sort of intuition that comes from knowledge of regular neighborhood theory, as explained in \cite[§3]{RS72} or \cite[§II]{Hu69}. For example, the reason that collapsing the boundary of $Sd^k(\Delta [2])$ in the category $nsSet$ is an operation that preserves the homotopy type in the case when $k=2$, but not in the case when $k=1$ is indicated and illustrated in a remark in \cite[p.~51]{Hu69}. It turns out that the double subdivision creates a sufficiently nice neighborhood around the boundary. \cref{fig:ch2_Desing_doublysubd_2simpl}, which is used for bookkeeping in the proof of \cref{prop:double_subdivision_sphere_low_dimension}, illustrates the phenomenon.  

\begin{figure}
\centering
\begin{tikzpicture}

\coordinate (2) at (90:3cm);
\coordinate (0) at (210:3cm);
\coordinate (1) at (-30:3cm);

\node [above] at (2) {2};
\node [below left] at (0) {0};
\node [below right] at (1) {1};

\draw [dashed] (0.north east)--(2.south) coordinate[midway](02);
\draw [dashed] (0.north east)--(1.north west) coordinate[midway](01);
\draw [dashed] (1.north west)--(2.south) coordinate[midway](12);

\coordinate (012) at (barycentric cs:0=1,1=1,2=1);

\draw (0.north east)--(012) coordinate[midway](0<012);
\draw (1.north west)--(012) coordinate[midway](1<012);
\draw (2.south)--(012) coordinate[midway](2<012);
\draw (01)--(012) coordinate[midway](01<012);
\draw (12)--(012) coordinate[midway](12<012);
\draw (02)--(012) coordinate[midway](02<012);

\coordinate (0<01<012) at (barycentric cs:0=1,01=1,012=1);
\coordinate (0<01) at (barycentric cs:0=0.5,01=0.5);
\foreach \x in {(0),(0<01),(01),(012)}
	\draw (0<01<012)--\x;
\foreach \x in {(0<012),(01<012)}
	\draw [thick] (0<01<012)--\x;

\coordinate (1<01<012) at (barycentric cs:1=1,01=1,012=1);
\coordinate (1<01) at (barycentric cs:1=0.5,01=0.5);
\foreach \x in {(1),(1<01),(01),(012)}
	\draw (1<01<012)--\x;
\foreach \x in {(01<012),(1<012)}
	\draw [thick] (1<01<012)--\x;

\coordinate (1<12<012) at (barycentric cs:1=1,12=1,012=1);
\coordinate (1<12) at (barycentric cs:1=0.5,12=0.5);
\foreach \x in {(1),(1<12),(12),(012)}
	\draw (1<12<012)--\x;
\foreach \x in {(12<012),(1<012)}
	\draw [thick] (1<12<012)--\x;

\coordinate (2<12<012) at (barycentric cs:2=1,12=1,012=1);
\coordinate (2<12) at (barycentric cs:2=0.5,12=0.5);
\foreach \x in {(2),(2<12),(12),(012)}
	\draw (2<12<012)--\x;
\foreach \x in {(12<012),(2<012)}
	\draw [thick] (2<12<012)--\x;
	
\coordinate (2<02<012) at (barycentric cs:2=1,02=1,012=1);
\coordinate (2<02) at (barycentric cs:2=0.5,02=0.5);
\foreach \x in {(2),(2<02),(02),(012)}
	\draw (2<02<012)--\x;
\foreach \x in {(02<012),(2<012)}
	\draw [thick] (2<02<012)--\x;
	
\coordinate (0<02<012) at (barycentric cs:0=1,02=1,012=1);
\coordinate (0<02) at (barycentric cs:0=0.5,02=0.5);
\foreach \x in {(0),(0<02),(02),(012)}
	\draw (0<02<012)--\x;
\foreach \x in {(02<012),(0<012)}
	\draw [thick] (0<02<012)--\x;

\foreach \x in {0.2,0.4,0.6,0.8}
	\draw [dotted] (barycentric cs:0=\x,0<01<012=1-\x)--(barycentric cs:01=\x,0<01<012=1-\x);

\foreach \x in {0.2,0.4,0.6,0.8}
	\draw [dotted] (barycentric cs:1=\x,1<01<012=1-\x)--(barycentric cs:01=\x,1<01<012=1-\x);

\foreach \x in {0.2,0.4,0.6,0.8}
	\draw [dotted] (barycentric cs:1=\x,1<12<012=1-\x)--(barycentric cs:12=\x,1<12<012=1-\x);

\foreach \x in {0.2,0.4,0.6,0.8}
	\draw [dotted] (barycentric cs:2=\x,2<12<012=1-\x)--(barycentric cs:12=\x,2<12<012=1-\x);

\foreach \x in {0.2,0.4,0.6,0.8}
	\draw [dotted] (barycentric cs:2=\x,2<02<012=1-\x)--(barycentric cs:02=\x,2<02<012=1-\x);

\foreach \x in {0.2,0.4,0.6,0.8}
	\draw [dotted] (barycentric cs:0=\x,0<02<012=1-\x)--(barycentric cs:02=\x,0<02<012=1-\x);
\end{tikzpicture}
\caption{Desingularizing the double Kan subdivision of the standard $2$-simplex with collapsed boundary.}
\label{fig:ch2_Desing_doublysubd_2simpl}
\end{figure}
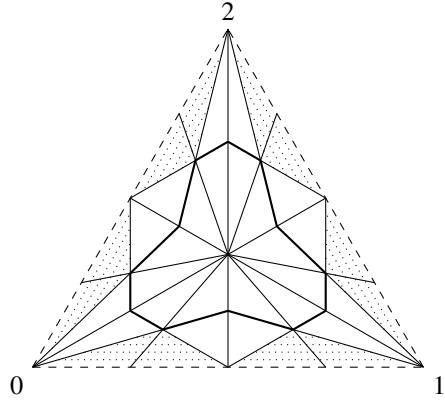

We are ready to prove the proposition. The method is similar to that of \cref{ex:desingularization_model_n-sphere_subdivided}.
\begin{proof}[Proof of \cref{prop:double_subdivision_sphere_low_dimension}.]
We will argue that
\begin{equation}
\label{eq:first_expression_proof_of_prop_double_subdivision_sphere_low_dimension}
DSd^2\, X\cong S(12-gon)
\end{equation}
where $X=\Delta [2]/\partial \Delta [2]$. By this, we mean that $DSd^2\, X$ is the suspension of a $12$-gon, which is what $BSd\, X$ is. As the cases when $n=0$ and $n=1$ were taken care of by \cref{ex:all_subdivisions_sphere_dim_zero_coincidence_dim_one} and \cref{ex:double_subdivision_sphere_dim_one}, respectively, the argument below finishes the proof.

To study $DSd^2\, X$ is to study the diagram that we get by applying $Sd$ to (\ref{eq:first_diagram_ex_desingularization_model_n-sphere_subdivided}). For an illustration of the formation of $DSd^2\, X$ from $Sd^2\, X$, see \cref{fig:ch2_Desing_doublysubd_2simpl}. We use the same conventions as in \cref{fig:ch2_Desing_singlysubd_2simpl} and one additional convention. Namely, there are exactly twelve line segments that are thicker than the others. These form the $12$-gon we mentioned. The simplicial set $BSd\, X$ is the nerve of the poset
\begin{equation}
\label{eq:first_diagram_proof_prop_double_subdivision_sphere_low_dimension}
\begin{gathered}
\xymatrix@=0.6em{
&&&&&& [012] \\
\\
\\
&&&&&& \bullet \ar@{..>}[lllld] \ar@{<-}[uuu] \ar@/^0.3pc/@{<..}[ddddddddd] \ar@{..>}[drrrr] \\
&& \bullet \ar@{<..}[ddddddddrrrr] \ar@{<-}[uuuurrrr] &&&&&&&& \bullet \ar@{<..}[lllldddddddd] \ar@{<-}[lllluuuu] \\
&& \bullet \ar@{..>}[lld] \ar@{<..}[dddddddrrrr] \ar@{..>}[u] \ar@{<-}[uuuuurrrr] &&&&&&&& \bullet \ar@{<-}[lllluuuuu] \ar@{..>}[u] \ar@{<..}[llllddddddd] \ar@{..>}[drr] \\
\bullet \ar@{<..}[ddddddrrrrrr] \ar@{<-}[uuuuuurrrrrr] &&&&&&&&&&&& \bullet \ar@{<-}[lllllluuuuuu] \ar@{<..}[lllllldddddd] \\
\bullet \ar[u] \ar@{<-}[dddddrrrrrr] \ar@{<-}[uuuuuuurrrrrr] \ar[drr] &&&&&& \bullet \ar[lllld] \ar@/^0.3pc/@{<-}[uuuuuuu] \ar[drrrr] &&&&&& \bullet \ar[lld] \ar@{<-}[llllllddddd] \ar@{<-}[lllllluuuuuuu] \ar[u] \\
&& \bullet \ar@{<-}[ddddrrrr] \ar@{<-}[uuuuuuuurrrr] &&&&&&&& \bullet \ar@{<-}[lllluuuuuuuu] \ar@{<-}[lllldddd] \\
\\
\\
\\
&&&&&& [0] \ar[uuuuu]
}
\end{gathered}
\end{equation}
namely $Sd(X)^\sharp$. In (\ref{eq:first_diagram_proof_prop_double_subdivision_sphere_low_dimension}) we have drawn the $0$-simplex $012$ as the cone point at the top.

The cone point at the bottom, which we denote $[0]$, is the $0$-simplex that is the result of the identifications
\[0\sim 01\sim 1\sim 12\sim 2\sim 02.\]
These names arise in an intuitive manner from considering the poset $\Delta [2]^\sharp$ whose objects correspond to the $0$-simplices of
\[B(\Delta [2])\cong Sd(\Delta [2])\]
whose non-degenerate simplices in turn correspond to the $0$-simplices of
\[B^2(\Delta [2])\cong BSd(\Delta [2])\cong Sd^2(\Delta [2]).\]
For example, the object $0$ arises from $\varepsilon _0$ and $1$ from $\varepsilon _1$. Furthermore, the object $02$ arises from $\delta _1$. The objects of the poset $Sd(X)^\sharp$ that are not cone points are the non-degenerate $1$-simplices of $Sd\, X$, of which there are six, and the non-degenerate $2$-simplices, of which there are also six.

We proceed by naming the twelve non-embedded non-degenerate $2$-simplices of $Sd^2\, X$. First, we let $y_1$ be the image of
\[\{ \{0\} <\{0<01\} <\{0<01<012\} \}\]
under $Sd^2(\Delta [2])\to Sd^2\, X$. Next, we let $y_2$ be the image of the next $2$-simplex as we move counterclockwise in \cref{fig:ch2_Desing_doublysubd_2simpl} up to and including $j=12$. Write $J=\{ 1,\dots ,12\}$. Each of the simplices $y_j$, $j\in J$, has the elementary degeneracy operator
\[\rho _{y_j}=\sigma _0\]
as its enforcer. Let $\rho$ denote this common enforcer.

From \cref{prop:role_of_enforcers}, we have the cocartesian square
\begin{equation}
\label{eq:second_diagram_proof_prop_double_subdivision_sphere_low_dimension}
\begin{gathered}
\xymatrix{
\bigsqcup _{j\in J}\Delta [2] \ar[d]_{\vee _{j\in J}(\bar{y} _j)} \ar[rr]^{\sqcup _{j\in J}(\rho )} && \bigsqcup _{j\in J}\Delta [1] \ar[d] \\
Sd^2\, X \ar[rr] && Z
}
\end{gathered}
\end{equation}
in $sSet$. Let $z_j$, $j\in J$, be the image of $y_j$ under $Sd^2\, X\to Z$. Suppose $z_j=w_j\rho$ for some simplex $w_j$, $j\in J$. Then $w_j$ is embedded as $Sd^2\, X\to Z$ is injective in degree $0$.

The elementary face operators $\delta _1$ and $\delta _0$ are both sections of $\rho$, so we have
\begin{displaymath}
\begin{array}{rclclcl}
z_j\delta _1 & = & (w_j\rho )\delta _1 & = & w_j(\rho \delta _1) & = & w_j \\
z_j\delta _0 & = & (w_j\rho )\delta _0 & = & w_j(\rho \delta _0) & = & w_j.
\end{array}
\end{displaymath}
for each $j\in J$. It follows that the image under $Sd^2\, X\to Z$ of each of the faces $y_j\delta _1$ and $y_j\delta _0$ is equal to $w_j$. Let us express this with $y_j\delta _1\sim y_j\delta _0$.

Suppose $j\in J$ odd. Then
\[y_j\delta _1\sim y_j\delta _0=y_{j+1}\delta _0\sim y_{j+1}\delta _1.\]
Thus we observe that $w_j=w_{j+1}$. We get that $Z$ is non-singular by the bookkeeping performed with the aid of \cref{fig:ch2_Desing_doublysubd_2simpl}. From \cref{lem:pushout_along_enforcers_intermediate_desingularization}, it follows that the simplicial set $Z$ is the desingularization of $Sd^2\, X$. Moreover, the simplicial set $Z$ is the nerve of (\ref{eq:first_diagram_proof_prop_double_subdivision_sphere_low_dimension}). The naturality of $t_{Sd\, X}$ shows that it is an isomorphism.
\end{proof}

%% file: sections/descr.tex
\section{Iterative description}
\label{sec:description}

In the appendix of his PhD thesis, Gaunce Lewis Jr. \cite[p.~158]{Le78} makes explicit the least drastic way of transforming a $k$-space into a compactly generated space, which is (defined as) a space that is both a $k$-space and a weak Hausdorff space. Lewis describes an iterative process. At each stage of the process, two points are identified whenever it is impossible to separate them by (disjoint) open sets.

We will provide an iterative description of the process of forming $UDX$ from $X$ that is analogous to Lewis' method. In the least drastic way possible, we want to make a quotient of $X$ so that the vertices of any non-degenerate simplex are pairwise distinct. In other words, any non-degenerate simplex of $X$ whose vertices are not pairwise distinct, must be made degenerate. For this purpose, we will use the notion of enforcer from \cref{def:enforcer}.

In relation to \cref{thm:main_result_itdesing}, there is a systematic study of reflective subcategories provided by S. Baron \cite{Ba69}. First, $nsSet$ is epi-reflective as the map $X\to DX$ is epic in general. Second, Baron discusses the possibility of factoring the reflector through a unique intermediate category.

In the following way, we define a functor $J:sSet\to sSet$ together with a natural quotient map $X\to JX$ that $\eta _X$ factors through. The functor $J$ is thought of as a preferred first step towards making a simplicial set non-singular. We have taken the symbol $J$ because Lewis uses it to denote his analogous endofunctor of $k$-spaces.

Let $X$ be a simplicial set. Given a non-degenerate simplex $x$ of $X$, we let $n_x$ denote its degree. Recall the enforcer $\rho _x:[n_x]\to [m_x]$ of $x$ from \cref{def:enforcer}. We will construct a cobase change of
\[A=\bigsqcup _{x\in X^\sharp}\Delta [n_x]\xrightarrow{f=\sqcup _{x\in X^\sharp }(\rho _x)} \bigsqcup _{x\in X^\sharp}\Delta [m_x]=B\]
along
\[A\xrightarrow{g=\vee _{x\in X^\sharp }(\bar{x} )} X.\]
The latter map is degreewise surjective as $X^\sharp$ generates $X$.

For each integer $n\geq 0$, define a symmetric binary relation $R'_n$ on $X_n$ by letting $(x,x')\in X_n\times X_n$ be a member of a set $R'_n$ if there are $a,a'\in A_n$ such that
\begin{displaymath}
\begin{array}{rcl}
x & = & g(a) \\
x' & = & g(a') \\
f(a) & = & f(a').
\end{array}
\end{displaymath}
The binary relation $R'_n$, $n\geq 0$, is reflexive as $g$ is degreewise surjective.

Let $R_n$ be the equivalence relation generated by $R'_n$, for each $n$. It follows immediately that the equivalence relations $R_n$, $n\geq 0$, satisfy the condition that the diagrams (\ref{eq:first_diagram_proof_lem_desing_unique_factorization_through_unit}) commute. This implies that we can form the quotient
\[JX=X/R.\]
Thus we obtain the cocartesian square
\begin{displaymath}
 \xymatrix{
 \bigsqcup _{x\in X^\sharp}\Delta [n_x] \ar[d]_{\vee _{x\in X^\sharp }(\bar{x} )} \ar[rr]^{\sqcup _{x\in X^\sharp }(\rho _x)} && \bigsqcup _{x\in X^\sharp}\Delta [m_x] \ar[d] \\
 X \ar[rr] && JX
 }
\end{displaymath}
in $sSet$. By \cref{lem:pushout_along_enforcers_intermediate_desingularization}, it gives rise to a commutative triangle
\begin{equation}
\label{eq:first_diagram_description_iterative}
\begin{gathered}
\xymatrix{
X \ar[dr]_{\eta _X} \ar[rr] && JX \ar[ld] \\
& UDX
}
\end{gathered}
\end{equation}
that factors the unit $\eta _X$ through a quotient map $X\to JX$, which is the identity in the case when $X$ is already non-singular.

For the purposes of making an iterative description of desingularization, the notation above is suitable. However, the construction $JX$ deserves its own name.
\begin{definition}\label{def:enforced_collapse}
Let $X$ be a simplicial set. The map $X\to JX$ is the \textbf{enforced collapse} of $X$.
\end{definition}
\noindent Outside of the context of the iteration process below we may choose to use the following symbol
\begin{notation}
Let $X$ be a simplicial set. Let
\[Cen(X)=JX\]
denote the enforced collapse of $X$.
\end{notation}
\noindent Note that the enforced collapse need not be non-singular, as \cref{ex:infinitely_many_enforced_collapses_needed} shows.
\begin{example}\label{ex:infinitely_many_enforced_collapses_needed}
Consider the $2$-dimensional simplicial set depicted in \cref{fig:ch2_finite_number_of_enforced_collapses_not_enough}. Identify the two $0$-simplices $v$ and $w$. The result can be constructed thus.

Let
\[\mathbb{N} =\{ 1,2,\dots \}\]
and
\[\mathbb{N} _0=\{ 0,1,2,\dots \} .\]
Next, for each $n\in 2\mathbb{N}$, let $B_n=\Delta [2]$. For each $n\in \mathbb{N} _0$, let $A_n=\Delta [1]$. Furthermore, let $C_0=\Delta [1]/\partial \Delta [1]$. Finally, for each $n\in \mathbb{N}$, let $C_n=\Delta [1]$.

Take the pushout $X$ in $sSet$ of
\begin{equation}
\label{eq:first_diagram_ex_infinitely_many_enforced_collapses_needed}
\begin{gathered}
\xymatrix{
\bigsqcup _{n\in \mathbb{N} _0}A_n \ar[d] \ar[r] & \bigsqcup _{n\in 2\mathbb{N} _0}C_n \\
\bigsqcup _{n\in 2\mathbb{N} }B_n
}
\end{gathered}
\end{equation}
where the maps are defined as follows. Let $X$ denote the pushout.

Suppose $n\in \mathbb{N} _0$. In the case when $n\equiv 0\; (mod\, 4)$, we let $A_n\to B_{n+2}$ be the map induced by $\delta _1$ and we let $A_{n+1}\to B_{n+2}$ be the map induced by $\delta _2$. In the case when $n\equiv 2\; (mod\, 4)$, we let $A_n\to B_{n+2}$ be the map induced by $\delta _1$ and we let $A_{n+1}\to B_{n+2}$ be the map induced by $\delta _0$. These maps give rise to the map
\[\bigsqcup _{n\in \mathbb{N} _0}A_n\to \bigsqcup _{n\in 2\mathbb{N} }B_n\]
in (\ref{eq:first_diagram_ex_infinitely_many_enforced_collapses_needed}).

Let $A_0\to C_0$ be the canonical map. Suppose $n\in \mathbb{N} _0$ odd. Then we let $A_n\to C_{n+1}$ and $A_{n+1}\to C_{n+1}$ be the identity $\Delta [1]\to \Delta [1]$. These maps give rise to the map
\[\bigsqcup _{n\in \mathbb{N} _0}A_n\to \bigsqcup _{n\in 2\mathbb{N} _0}C_n\]
in (\ref{eq:first_diagram_ex_infinitely_many_enforced_collapses_needed}).

If $Cen^k$ denotes the $k$-fold iteration of the enforced collapse for $k$ a non-negative integer, then $Cen^k(X)$ is singular for every $k$.
\end{example}
\noindent \cref{ex:infinitely_many_enforced_collapses_needed} shows that one might need an infinite number of enforced collapses in order to make a simplicial set non-singular.

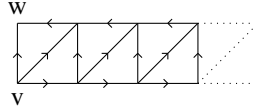
\begin{figure}
\centering
\begin{tikzpicture}[scale=0.4]

\coordinate (0) at (-5,-1);
\coordinate (2) at (-3,-1);
\coordinate (4) at (-1,-1);
\coordinate (6) at (1,-1);

\coordinate (1) at (-5,1);
\coordinate (3) at (-3,1);
\coordinate (5) at (-1,1);
\coordinate (7) at (1,1);

\node [below] at (0) {v};

\node [above] at (1) {w};

\draw (0.north)--(1.south) coordinate[midway](01);
\draw (2.north)--(3.south) coordinate[midway](23);
\draw (4.north)--(5.south) coordinate[midway](45);
\draw (6.north)--(7.south) coordinate[midway](67);

\draw (0.north)--(3.south) coordinate[midway](03);
\draw (2.north)--(5.south) coordinate[midway](25);
\draw (4.north)--(7.south) coordinate[midway](47);
\draw [dotted] (6.north)--(3,1);

\draw (0.north)--(2.north) coordinate[midway](02);
\draw (2.north)--(4.north) coordinate[midway](24);
\draw (4.north)--(6.north) coordinate[midway](46);
\draw [dotted] (6.north)--(3,-1);

\draw (1.south)--(3.south) coordinate[midway](13);
\draw (3.south)--(5.south) coordinate[midway](35);
\draw (5.south)--(7.south) coordinate[midway](57);
\draw [dotted] (7.south)--(3,1);

\draw (01)--+(225:0.2cm);
\draw (01)--+(315:0.2cm);

\draw (23)--+(225:0.2cm);
\draw (23)--+(315:0.2cm);

\draw (45)--+(225:0.2cm);
\draw (45)--+(315:0.2cm);

\draw (67)--+(225:0.2cm);
\draw (67)--+(315:0.2cm);

\draw (02)--+(135:0.2cm);
\draw (02)--+(225:0.2cm);

\draw (24)--+(135:0.2cm);
\draw (24)--+(225:0.2cm);

\draw (46)--+(135:0.2cm);
\draw (46)--+(225:0.2cm);

\draw (13)--+(45:0.2cm);
\draw (13)--+(315:0.2cm);

\draw (35)--+(45:0.2cm);
\draw (35)--+(315:0.2cm);

\draw (57)--+(45:0.2cm);
\draw (57)--+(315:0.2cm);

\draw (03)--+(180:0.2cm);
\draw (03)--+(270:0.2cm);

\draw (25)--+(180:0.2cm);
\draw (25)--+(270:0.2cm);

\draw (47)--+(180:0.2cm);
\draw (47)--+(270:0.2cm);
\end{tikzpicture}
\caption{Simplicial set such that every finite iteration of enforced collapses is singular.}
\label{fig:ch2_finite_number_of_enforced_collapses_not_enough}
\end{figure}

We point out the following, which is not really part of the storyline.
\begin{remark}\label{rem:description_iterative}
The map $\vee _{x\in X^\sharp }(\bar{x} )$ is degreewise surjective because $X^\sharp$ generates $X$. In this way, the construction of the functor $J$ is less arbitrary than the setting in \cref{lem:pushout_along_enforcers_intermediate_desingularization}.

One can, however, replace $X^\sharp$ with a subset and still construct symmetric binary relations $R'_n$, $n\geq 0$, the same way. Each of them is reflexive if and only if the subset generates $X$. We can in either case choose a quotient map as the cobase change of $\sqcup _{x\in X^\sharp }(\rho _x)$ along $\vee _{x\in X^\sharp }(\bar{x} )$.

For example, in the proof of \cref{prop:double_subdivision_sphere_low_dimension}, or more specifically the diagram (\ref{eq:second_diagram_proof_prop_double_subdivision_sphere_low_dimension}), we did choose a suitable subset of the set of non-degenerate simplices to perform a desingularization.
\end{remark}
\noindent \cref{rem:description_iterative} might be useful in some cases as suggested by the proof of \cref{prop:double_subdivision_sphere_low_dimension}.

To define $J$ on morphisms $f:X\to Y$ we need a diagram of the form
\begin{equation}
\label{eq:second_diagram_description_iterative}
\begin{gathered}
 \xymatrix{
 X \ar[d]_f & \bigsqcup _{x\in X^\sharp}\Delta [n_x] \ar[l] \ar[d] \ar[r] & \bigsqcup _{x\in X^\sharp}\Delta [m_x] \ar@{-->}[d] \\
 Y & \bigsqcup _{y\in Y^\sharp}\Delta [n_y] \ar[l] \ar[r] & \bigsqcup _{y\in Y^\sharp}\Delta [m_y]
 }
\end{gathered}
\end{equation}
in which an obvious choice of middle vertical map is $f(x)^\flat$ for each index $x\in X^\sharp$. Here, we write $f(x)=f(x)^\sharp f(x)^\flat$ by means of the Eilenberg-Zilber lemma.

There is at most one dashed map that makes the square
\begin{equation}
\label{eq:third_diagram_description_iterative}
\begin{gathered}
 \xymatrix{
 [n_x] \ar[d]_{f(x)^\flat } \ar[r]^{\rho _x} & [m_x] \ar@{-->}[d] \\
 [n_{f(x)^\sharp }] \ar[r]_{\rho _{f(x)^\sharp }} & [m_{f(x)^\sharp }]
 }
\end{gathered}
\end{equation}
commute as $\rho _x$. We claim that if $\mu _x$ is a section of $\rho _x$, then
\[\rho _{f(x)\sharp }\circ f(x)^\flat \circ \mu _x\]
makes the square commute. This claim holds if
\begin{equation}\label{eq:first_equation_description_iterative}
\rho _{f(x)^\sharp }\circ f(x)^\flat (i)=\rho _{f(x)^\sharp }\circ f(x)^\flat (j)
\end{equation}
whenever
\begin{equation}\label{eq:second_equation_description_iterative}
\rho _x(i)=\rho _x(j).
\end{equation}
If the claim holds, then any other section of $\rho _x$ would yield the same functor $[k_x]\to [k_{f(x)^\sharp }]$. From dashed maps that makes the diagrams (\ref{eq:third_diagram_description_iterative}) commute, we get a dashed map that makes (\ref{eq:second_diagram_description_iterative}) commute. With it arises a map $J(f)$.

Now we argue that (\ref{eq:first_equation_description_iterative}) holds whenever (\ref{eq:second_equation_description_iterative}) does. The degeneracy operator $\rho _x$ corresponds to the equivalence relation on $[n_x]$ that is generated by the reflexive binary relation $\approx$ that is defined in \cref{sec:calculations}. Hence, our claim will follow if $i\approx k$ implies that
\begin{equation}\label{eq:third_equation_description_iterative}
\rho _{f(x)^\sharp }\circ f(x)^\flat (i)=\rho _{f(x)^\sharp }\circ f(x)^\flat (k)
\end{equation}
holds.

Suppose $x\varepsilon _i=x\varepsilon _j$. This implies $f(x)\varepsilon _i=f(x)\varepsilon _j$, which can be rewritten as
\[f(x)^\sharp f(x)^\flat \varepsilon _i=f(x)^\sharp f(x)^\flat \varepsilon _j,\]
which in turn can be rewritten as
\[f(x)^\sharp \varepsilon _{f(x)^\flat (i)}=f(x)^\sharp \varepsilon _{f(x)^\flat (j)}.\]
By definition of $\rho _{f(x)^\sharp }$ it follows that
\[\rho _{f(x)^\sharp }(f(x)^\flat (i))=\rho _{f(x)^\sharp }(f(x)^\flat (j)).\]
Next, suppose $i\leq k\leq j$. In other words, we assume $i\approx k$. Degeneracy operators are order-preserving, so (\ref{eq:third_equation_description_iterative}) holds. This concludes our definition of $J(f)$.

It is clear that $J(id)=id$, for in the case $f=id$ we have that $f(x)^\flat =id$ and $\rho _x=\rho _{f(x)^\sharp }$. It follows that
\[J(g\circ f)=J(g)\circ J(f)\]
from the fact that the square
\begin{displaymath}
\xymatrix{
X \ar[d] \ar[r]^f & Y \ar[d] \\
JX \ar[r]_{J(f)} & JY
}
\end{displaymath}
commutes for each simplicial map $f:X\to Y$ combined with the fact that $X\to JX$ is degreewise surjective for each simplicial set $X$. Thus the construction $JX$ is functorial and the map $X\to JX$ is natural. Because $X\to JX$ is natural and degreewise surjective and because $\eta _X$ is natural, it follows that $JX\to UDX$ is natural.

The plan is to obtain a quotient of $X$ that is isomorphic to $UDX$ by applying $J$ successively. Moreover, we aim to establish \cref{thm:main_result_itdesing}. To arrange for the iteration, we refer to \cref{def:sequence_composition}. Let $f^{0,1}$ be the natural map
\[J^0X=X\to JX=J^1X.\]
Due to (\ref{eq:first_diagram_description_iterative}), we can assume that we for some ordinal $\gamma >1$ have defined a $\gamma$-sequence
\[T^{[0]}\Rightarrow \cdots \Rightarrow T^{[\beta ]}\Rightarrow T^{[\beta +1]}\Rightarrow \cdots\]
of commutative triangles
\begin{equation}
\label{eq:fourth_diagram_description_iterative}
\begin{gathered}
\xymatrix{
X \ar[dr]_{\eta _X} \ar[rr] && J^\beta X \ar[ld]^{p^\beta} \\
& UDX
}
\end{gathered}
\end{equation}
denoted $T^{[\beta ]}$ and natural transformations, in which the component
\[J^\alpha X\xrightarrow{f^{\alpha ,\beta }} J^\beta X\]
of $T^{[\alpha ]}\Rightarrow T^{[\beta ]}$ is a quotient map whenever $\alpha \leq \beta <\gamma$.

If $\gamma$ is a limit ordinal, then we take the colimit in the following way to define $J^\gamma X$. For each $n\geq 0$, let $R_n$ be the equivalence relation on $J^0X=X$ that consists of the elements $(x,y)\in X_n\times X_n$ such that there is some $\beta <\gamma$ with $f^{0,\beta }(x)=f^{0,\beta }(y)$. It is clear that the diagrams (\ref{eq:first_diagram_proof_lem_desing_unique_factorization_through_unit}) commute so that we obtain the quotient $J^\gamma X=X/R$ of $J^0X$. In this case, we automatically get a diagram $T^{[\gamma ]}$ that plays the role of (\ref{eq:fourth_diagram_description_iterative}).

Else if $\gamma =\beta +1$ is a successor of an ordinal $\beta$, then we simply define $J^{\beta +1}X$ by applying $J$ to $J^\beta X$. Consider the solid commutative diagram
\begin{equation}
\label{eq:fifth_diagram_description_iterative}
\begin{gathered}
\xymatrix{
X \ar[dd]_{id} \ar[dr]_{\eta _X} \ar[rr] && J^\beta X \ar[ld]^{p^\beta } \ar[dd]^{f^{\beta ,\beta +1}} \\
& UDX \ar[dd]_(.65){id} \\
X \ar[dr]_{\eta _X} \ar@{-}[r] & \ar[r] & J^{\beta +1}X \ar@{-->}[ld]^{p^{\beta +1}} \\
& UDX
}
\end{gathered}
\end{equation}
in which we have yet to define the dashed map $p^{\beta +1}$. By \cref{prop:role_of_enforcers}, we obtain the dashed map in the solid diagram
\begin{equation}
\label{eq:sixth_diagram_description_iterative}
\begin{gathered}
\xymatrix{
& \bigsqcup _{j\in (J^\beta X)^\sharp}\Delta [n_j] \ar[dd]_{\vee _{j\in (J^\beta X)^\sharp }(\bar{\jmath } )} \ar[rr]^{\sqcup _{j\in (J^\beta X)^\sharp }(\rho _j)} && \bigsqcup _{j\in (J^\beta X)^\sharp}\Delta [m_j] \ar[dd] \\
\\
X \ar[d]_\eta \ar[r]^{f^{0,\beta }} & J^\beta X \ar[ld]^(.45){p^\beta } \ar[d]^\eta \ar[rr]^{f^{\beta ,\beta +1}} && J^{\beta +1}X \ar@{-->}[lld] \ar[d]^\eta \\
UDX \ar[r]_(.45){UD(f^{0,\beta })} & UD(J^\beta X) \ar[rr]_{UD(f^{\beta ,\beta +1})} && UD(J^{\beta +1})
}
\end{gathered}
\end{equation}
in $sSet$, which commutes because $f^{0,\beta }$ is a quotient map and hence degreewise surjective.

The whole diagram (\ref{eq:sixth_diagram_description_iterative}) commutes because $f^{\beta ,\beta +1}$ is degreewise surjective. This implies that $UD(f^{0,\beta })$ and $UD(f^{\beta ,\beta +1})$ are isomorphisms. Hence, from (\ref{eq:sixth_diagram_description_iterative}) we obtain a canonical dashed map $p^{\beta +1}$ in (\ref{eq:fifth_diagram_description_iterative}) that makes the whole diagram commute, including the lower triangle.

We have finished the construction of a $\gamma$-sequence $T:\gamma \to sSet^{[2]}$ for each ordinal $\gamma$. By the design of these sequences, there is a canonical composition of each of them that is a quotient map.

Next, we verify that this iterative process does indeed come to a halt. The proof uses the following observation.
\begin{lemma}\label{lem:no_break_iterative_desingularization}
If $\beta$ is some ordinal and if some $x\in (J^\beta X)^\sharp$ is not embedded, then $f^{\beta ,\beta +1}(x)$ is a degenerate simplex in $J^{\beta +1}X$.
\end{lemma}
\begin{proof}
Consider the diagram
\begin{equation}
\label{eq:first_diagram_proof_lem_no_break_iterative_desingularization}
\begin{gathered}
\xymatrix@=0.7em{
\Delta [n_x] \ar[dd] \ar[rr]^{\rho _x} && \Delta [m_x] \ar@/_1.9pc/[lddd] \ar[dd] \\
\\
\bigsqcup _{j\in (J^\beta X)^\sharp }\Delta [n_j] \ar[dd]_{\vee _{j\in (J^\beta X)^\sharp }{(\bar{\jmath } )}} \ar@{-}[r] & \ar[r] & \bigsqcup _{j\in (J^\beta X)^\sharp }\Delta [m_j] \ar[dd] \\
& Y \ar@{-->}[dr] \\
J^\beta X \ar[ur] \ar[rr]_{f^{\beta ,\beta +1}} && J^{\beta +1}X
}
\end{gathered}
\end{equation}
where we take the pushout
\[Y=J^\beta X\sqcup _{\Delta [n_x]}\Delta [m_x].\]
The quotient map $f^{\beta ,\beta +1}X$ factors through the canonical map $J^\beta X\to Y$. The map $Y\to J^{\beta +1}X$ is then also degreewise surjective. To say that $x$ is not embedded is the same as saying that its vertices are not pairwise distinct, so $\rho _x$ is a proper degeneracy operator. Thus we see that
\[\Delta [n_x]\xrightarrow{\bar{x} } J^\beta X\to Y\]
is the representing map of a degenerate simplex. To precompose this representing map with $Y\to J^{\beta +1}X$ yields the map $f^{\beta ,\beta +1}\circ \bar{x}$, as we see from (\ref{eq:first_diagram_proof_lem_no_break_iterative_desingularization}). It follows that $f^{\beta ,\beta +1}(x)$ is degenerate.
\end{proof}
\begin{proposition}\label{prop:iterative_desingularization_halts}
Let $X$ be a simplicial set. There is an ordinal $\lambda$ such that $J^\lambda X$ is non-singular.
\end{proposition}
\begin{corollary}\label{cor:iterative_desingularization_halts}
Let $X$ be a simplicial set. There is an ordinal $\lambda$ such that the map
\[p^\lambda :J^\lambda X\xrightarrow{\cong } UDX\]
is an isomorphism.
\end{corollary}
\begin{proof}[Proof of \cref{cor:iterative_desingularization_halts}.]
Use \cref{prop:iterative_desingularization_halts} to choose an ordinal $\kappa$ such that $J^\kappa X$ is non-singular.

According to \cref{lem:pushout_along_enforcers_intermediate_desingularization}, the canonical map $J^{\kappa +1}X\xrightarrow{\cong } UD(J^\kappa X)$ is an isomorphism as $J^{\kappa +1}X$ is non-singular, which is in turn because $f^{\kappa ,\kappa +1}$ is the identity. Recall the successor ordinal step from the construction of $T$ and replace $\beta$ with $\kappa$ in the diagram (\ref{eq:sixth_diagram_description_iterative}).

As $f^{\kappa ,\kappa +1}$ is the identity, it follows that the isomorphism above is in fact equal to $\eta _{J^\kappa X}$. The map $J^{\kappa +1}X\to UDX$ is by design equal to the composite
\[J^{\kappa +1}X=J^\kappa X\xrightarrow{\eta _{J^\kappa X}} UD(J^\kappa X)\to UDX.\]
The first half $\eta _{J^\kappa X}$ of the composite above is an isomorphism by the choice of $\kappa$ and the second half is the inverse of
\[UD(f^{0,\kappa }):UDX\to UD(J^\kappa X)\]
If we define $\lambda =\kappa +1$, then the proof is finished.
\end{proof}
\begin{proof}[Proof of \cref{prop:iterative_desingularization_halts}.]
The idea of the proof is that we can index the simplicial sets $J^\beta X$ that are singular by a certain subset of the non-degenerate simplices of $X$.

If $J^0X=X$ is already non-singular, then we can let $\lambda =0$. Else if $X$ is singular, then we choose a non-embedded non-degenerate simplex $x^0$ of $X$. Suppose $\gamma >0$ is such that we for all $\beta$ with $\beta <\gamma$ have defined $x^\beta$ with $x^\alpha \neq x^\beta$ if $\alpha <\beta <\gamma$.

If $J^\gamma X$ is non-singular, then we define $\lambda =\gamma$. Else if $J^\gamma X$ is singular, then we choose a simplex $x^\gamma$ of $X$ such that $f^{0,\gamma }(x^\gamma )$ is a non-embedded non-degenerate simplex. Suppose $\beta$ an ordinal with $\beta <\gamma$. From the commutative diagram
\begin{equation}
\label{eq:first_diagram_proof_prop_iterative_desingularization_halts}
\begin{gathered}
 \xymatrix{
 X \ar[dr]_{f^{0,\beta }} \ar[rr]^{f^{0,\gamma }} && J^\gamma X \\
 & J^\beta X \ar[ur]^{f^{\beta ,\gamma }} \ar[rr]_{f^{\beta ,\beta +1}} && J^{\beta +1}X \ar[lu]_{f^{\beta +1,\gamma }}
 }
 \end{gathered}
\end{equation}
we will conclude that
\begin{equation}\label{eq:first_equation_proof_prop_iterative_desingularization_halts}
x^\beta \neq x^\gamma
\end{equation}
in the following way.

Define
\begin{displaymath}
\begin{array}{rcl}
y & = & f^{0,\beta }( x^\beta) \\
y' & = & f^{\beta ,\beta +1}(y) \\

\end{array}
\end{displaymath}
As $y'$ is degenerate by \cref{lem:no_break_iterative_desingularization}, it follows that $f^{\beta +1,\gamma }(y')$ is degenerate. Because the diagram (\ref{eq:first_diagram_proof_prop_iterative_desingularization_halts}) commutes, this simplex is equal to
\[f^{\beta +1,\gamma }(y')=f^{\beta ,\gamma }(y)=f^{0,\gamma }(x^\beta ).\]
On the other hand, the simplex $f^{0,\gamma }(x^\gamma )$ is non-degenerate, so, as announced, it follows that (\ref{eq:first_equation_proof_prop_iterative_desingularization_halts}) holds.

Let $\lambda$ be a cardinal that is strictly greater than the cardinality of $X^\sharp$. Define $S$ as the set consisting of those $x^\beta$ with $\beta \leq \lambda$. This is a subset of $X^\sharp$. Then we can consider the injective function $S\to \lambda +1$ defined by $x^\beta \mapsto \beta$. If $\alpha <\beta$, then $x^\alpha$ is defined if $x^\beta$ is. In other words, $\alpha$ is in the image of $S\to \lambda +1$ if $\beta$ is.

By the choice of $\lambda$, there does not exist a surjective extension
\begin{displaymath}
\xymatrix{
S \ar[dd] \ar[dr] \\
& \lambda +1 \\
X^\sharp \ar@{-->>}[ur]_\nexists
}
\end{displaymath}
of $S\to \lambda +1$ to $X^\sharp$. Therefore, the function $S\to \lambda +1$ cannot possibly be surjective. Hence, the element $\lambda$ is not in the image of the latter function. By the definition of $S$, it follows that $x^\lambda$ is not defined. This implies that the set $S$ contains every element in $X^\sharp$ with a designation $x^\beta$. This shows that $J^\lambda X$ is non-singular.
\end{proof}
\begin{proof}[Proof of \cref{thm:main_result_itdesing}.]
Use \cref{cor:iterative_desingularization_halts} to choose an ordinal $\lambda$ such that $p^\lambda$ is an isomorphism. Take the corresponding $\lambda$-sequence $T$ of triangles (\ref{eq:fourth_diagram_description_iterative}) from the family of sequences constructed above. The map $f^{0,\lambda }$ is the composition of the $\gamma$-sequence
\[J^0X\xrightarrow{f^{0,1}} \cdots \to J^\beta X\xrightarrow{f^{\beta ,\beta +1}} \cdots\]
by the design of $J^{\lambda +1}$. Because $p^\lambda$ is an isomorphism, the commutative triangle $T^{[\lambda ]}$ identifies $f^{0,\lambda }$ with $\eta _X$.
\end{proof}